\documentclass[a4paper,12pt]{article}     % `article' for A4 paper
\usepackage{amsmath,amscd,amssymb}        % AMS-LaTeX with CD and symbols
\usepackage{latexsym}                     % old LaTeX symbols
\usepackage[english, german]{babel}                % German & English (default)
\usepackage{xypic}

\usepackage{theorem}                 % theorem extensions

\input{cyracc.def}

\newfont{\cyrfnt}{wncyr10}
\newcommand{\cyr}{\baselineskip12.5pt\cyrfnt\cyracc}

\newtheorem{theorem}{Theorem}
\newtheorem{corollary}{Corollary}
\newtheorem{proposition}{Proposition}

\theorembodyfont{\rmfamily}
\newtheorem{remark}{Remark}
\newtheorem{example}{Example}

\newenvironment{definition}
{\smallskip\noindent{\bf Definition\/}:}{\smallskip\par}

\newenvironment{proof}{\begin{ProofwCaption}{Proof}}{\end{ProofwCaption}}
\newenvironment{proof*}[1]{\begin{ProofwCaption}{{#1}}}{\end{ProofwCaption}}
\newenvironment{ProofwCaption}[1]%
  {\addvspace\theorempreskipamount \noindent{\it #1.}\rm}%
  {\qed \par \addvspace\theorempostskipamount}
\newcommand{\qedsymbol}{{\rm $\Box$}}
\newcommand{\qed}{\hfill\qedsymbol}
\newcommand{\CC}{{\mathbb C}}

\newcommand{\RR}{{\mathbb R}}
\newcommand{\ZZ}{{\mathbb Z}}

\newcommand{\xx}{{\mathbf x}}

\newcommand{\calH}{{\mathcal H}}

\newcommand{\calJ}{{\mathcal J}}

\newcommand{\calE}{{\mathcal E}}

\newcommand{\calG}{{\mathcal G}}
\newcommand{\calB}{{\mathcal B}}

\newcommand{\eps}{\varepsilon}

\newcommand{\ind}{{\rm ind}\,}

\title{An algebraic formula for the index of a 1-form on a real quotient singularity}
\author{Wolfgang Ebeling and Sabir M.~Gusein-Zade
\thanks{Partially supported by DFG. The work of the second author
(Sections~\ref{sect:Intro}, \ref{sect:equivariant}, \ref{sect:laws}, \ref{sect:one-dimrep}, and
\ref{sect:G-signature})
was supported by the grant 16-11-10018 of the Russian Science Foundation.
}
}
\date{}

\begin{document}
\selectlanguage{english}

\maketitle

\begin{abstract}
Let a finite abelian group $G$ act (linearly) on the space $\RR^n$ and thus on its complexification $\CC^n$.
Let $W$ be the real part of the quotient $\CC^n/G$ (in general $W \neq \RR^n/G$). We give an algebraic
formula for the radial index of a 1-form on the real quotient $W$. It is shown that this index
is equal to the signature of the restriction of the residue pairing to the $G$-invariant part
$\Omega^G_\omega$ of $\Omega_\omega= \Omega^n_{\RR^n,0}/\omega \wedge \Omega^{n-1}_{\RR^n,0}$.
%% For a $G$ invariant 1-form $\omega$, one has a natural analogue of the so-called quantum cohomology group
%% defined in the quantum singularity theory (FJRW-theory). We show that, for a real 1-form $\omega$,
%% the signature of the residue pairing on the real part of the quantum cohomology group is  equal to the
%% orbifold index of the 1-form $\omega$ on $\pi^{-1}(W)$.
For a $G$-invariant function $f$, one has the so-called quantum cohomology group
defined in the quantum singularity theory (FJRW-theory). We show that, for a real function $f$,
the signature of the residue pairing on the real part of the quantum cohomology group is equal to the
orbifold index of the 1-form $df$ on the preimage $\pi^{-1}(W)$ of $W$ under the natural quotient map.
\end{abstract}

%%%%%%%%%%%%%%%%%%%%%%%
\section{Introduction} \label{sect:Intro}
%%%%%%%%%%%%%%%%%%%%%%%
For an analytic map $F:(\RR^n,0) \to (\RR^n,0)$ such that $F_\CC^{-1}(0)=0$ ($F_\CC: (\CC^n,0) \to (\CC^n,0)$
is the complexification of $F$) one has the famous Eisenbud--Levine--Khimshiashvili algebraic formula for
its local degree: \cite{EL, Kh}. Let $F=(f_1, \ldots , f_n)$ and let
$Q_F:= \calE_{\RR^n,0}/\langle f_1, \ldots , f_n \rangle$, where $\calE_{\RR^n,0}$ is the ring of germs
of analytic functions on $(\RR^n,0)$. One has a natural residue pairing on $Q_F$. In \cite{EL, Kh} it is
shown that the degree of $F$ is equal to the signature of the residue pairing. This can also be
interpreted as a formula for the index of the singular point of the vector field
$\sum f_i \frac{\partial}{\partial x_i}$ or of the 1-form $\omega=\sum f_i dx_i$. Moreover, the choice
of a volume form permits to identify the algebra $Q_F$ (as a vector space) with the space
$\Omega_\omega= \Omega^n_{\RR^n,0}/\omega \wedge \Omega^{n-1}_{\RR^n,0}$. 

There exist notions of indices of vector fields and of 1-forms on singular varieties (see e.g.
\cite{BSS, EG-Sur}). One of them is the so-called radial index which is defined for a 1-form on an arbitrary
singular variety. X.~G\'omez-Mont and P.~Marde\v{s}i\'c gave an analogue of the signature formula for
the index of a vector field on an isolated (real) hypersurface singularity: \cite{GMM1,GMM2}.
Some formulae expressing the index of a 1-form on a real isolated complete intersection singularity
were given in \cite{EG-Fields, EG-MMJ}. However, the pairing in \cite{EG-Fields} was defined in
topological terms and the one in \cite{EG-MMJ} was defined in non-local terms on a deformation of
the singularity. Thus one can say that an algebraic signature formula for the index of a 1-form on
a real singular variety (say on a hypersurface singularity) which can be considered as an analogue
of those in \cite{EL, Kh, GMM1, GMM2} does not exist. In the framework of attempts to find such a formula,
there were defined (canonical) quadratic forms on analogues of the spaces $Q_F$ and $\Omega_\omega$
for 1-forms on isolated complete intersection singularities \cite{EGMZ}. However, these quadratic
forms appeared to be in general degenerate and relations of their signatures with indices remained unclear.

Equivariant (with respect to the action of a finite group $G$) versions of indices of $G$-invariant
1-forms were introduced in \cite{EG-EJM} as elements of the Burnside ring $A(G)$ of the group $G$.
In particular, there was defined the equivariant radial index. There was also introduced the notion
of an orbifold index of a $G$-invariant 1-form.

Let a finite abelian group $G$ act (linearly) on the space $\RR^n$ and thus on its complexification $\CC^n$.
Let $W$ be the real part of the quotient $\CC^n/G$. Note that in general $W \neq \RR^n/G$. A (real) 1-form
$\eta$ on $W$ defines a $G$-invariant (real) 1-form $\omega=\pi^\ast \eta$ on $\CC^n$ ($\pi: \CC^n \to \CC^n/G$
is the quotient map). The radial index of the 1-form $\eta$ is equal to the reduction under the group
homomorphism $r^{(0)}: A(G) \to \ZZ$ (see \cite{EG-EJM}) of the equivariant (radial) index of the
$G$-invariant 1-form $\omega$ on the preimage of $W$. Here we give an algebraic formula for the indicated
reduction of the equivariant index of a $G$-invariant 1-form on $\pi^{-1}(W)$ and thus an algebraic formula
for the radial index of a 1-form on the real quotient $W$. It is shown that this index is equal to
the signature of the restriction of the residue pairing to the $G$-invariant part $\Omega^G_\omega$ of
$\Omega_\omega$. 

For a germ $f$ of a quasihomogeneous function on $(\CC^n,0)$ with an isolated critical point at the origin
invariant with respect to an appropriate action of a finite abelian group $G$, H.~Fan, T.~Jarvis, and
Y.~Ruan \cite{FJR} defined the so-called quantum cohomology group. This group is considered as the main
object of the quantum singularity theory (FJRW-theory). An analogue of this group can be defined for an
arbitrary $G$-invariant function germ $f$.
%%% and, moreover, for a $G$-invariant 1-form $\omega$ on $(\CC^n,0)$.
Let us denote it by $\calH_{f,G}$. The vector space $\calH_{f,G}$ is the direct sum of the spaces
$(\Omega^\CC_{df^g})^G$ over the elements $g$ of the group $G$, where $f^g$ is the restriction of
the function $f$ to the fixed point set of $g$, $(\Omega^\CC_{df^g})^G$ is the $G$-invariant part
of the module $\Omega^\CC_{df^g}$. One has the residue pairing on each of the summands
$(\Omega^\CC_{df^g})^G$ and thus on the space $\calH_{f,G}$. If the function $f$ is real, one has
a natural real part $\calH^\RR_{f,G}$ of the space $\calH_{f,G}$ and the residue pairing is real
on it. We derive from the main result of the paper that the signature of the residue pairing on
$\calH^\RR_{f,G}$ is  equal to the orbifold index of the 1-form $df$ on $\pi^{-1}(W)$.

%%%%%%%%%%%%%%%%%%%%%%%
\section{Equivariant index of a 1-form} \label{sect:equivariant}
%%%%%%%%%%%%%%%%%%%%%%%
The Burnside ring $A(G)$ of a finite group $G$ is the Grothendieck ring of finite $G$-sets, see, e.g.,
\cite{Knutson}. As an abelian group, $A(G)$ is generated by the classes $[G/H]$ for subgroups $H$ of
the group $G$. For a topological space $X$ with a $G$-action, let $\chi^{(0)}(X):=\chi(X/G)$ be the Euler
characteristic of the quotient and let $\chi^{(1)}(X)=\chi^{\rm orb}(X,G)$ be the orbifold Euler
characteristic of the G-space $X$ (see, e.g., \cite{AS, HH}). Applying $\chi^{(0)}$ and $\chi^{(1)}$
to finite $G$-sets one gets group homomorphisms $r^{(0)}: A(G) \to \ZZ$ and $r^{(1)}: A(G) \to \ZZ$.
For an element $a=\sum\limits_{H\subset G} a_H [G/H] \in A(G)$ one has $r^{(0)}a=\sum\limits_{H\subset G} a_H$.
For an abelian group $G$,
one has $r^{(1)}[G/H]=|H|$. One has a natural ring homomorphism $r: A(G) \to R_\CC(G)$ from the Burnside
ring to the ring of (complex) representations of $G$ (sending a finite $G$-set to the vector space of
functions on it).

Let $(V,0) \subset (\RR^N,0)$ be a germ of a (real) subanalytic space with an action of a finite group $G$
on it, and let $\omega$ be a $G$-invariant 1-form on $(V,0)$ (that is, the restriction of a 1-form on
$(\RR^N,0)$) with an isolated singular point at the origin. In \cite{EG-EJM} there was defined
the equivariant (radial) index ${\rm ind}^G(\omega; V,0)$ as an element of the Burnside ring $A(G)$.

\begin{definition}
The {\em orbifold index} of the $G$-invariant 1-form $\omega$ is
\[
{\rm ind}^{\rm orb}(\omega; V,0):= r^{(1)} {\rm ind}^G(\omega; V,0). 
\]
\end{definition}

Let $(V/G,0)$ be the quotient of $V$ under the $G$-action and let $\eta$ be a 1-form on $(V/G,0)$ with
an isolated singular point at the origin.

\begin{proposition} \label{prop:indquot}
 One has
 \begin{equation}
  \ind(\eta;V/G,0)=r^{(0)}\ind^G(\pi^*\eta;V,0).
 \end{equation}
\end{proposition}

\begin{proof}
 For a point $x\in V$, let $G_x:=\{g\in G: gx=x\}$ be the isotropy subgroup of $x$, for a conjugacy class
 $[H]\in {\rm Conjsub\,}G$ of subgroups of $G$, let $V^{([H])}:=\{x\in V: G_x\in[H]\}$.
 The result \cite[Proposition~4.9]{EG-EJM} says that
 $$
 \ind^G(\pi^*\eta;V,0)=\sum_{[H]\in {\rm Conjsub\,}G}\ind(\eta;V^{([H])}/G,0)[G/H],
 $$
 where
 $$
 \sum_{[H]\in {\rm Conjsub\,}G}\ind(\eta;V^{([H])}/G,0)=\ind(\eta;V/G,0).
 $$
 (the statement in \cite[top of p.~290]{EG-EJM}).
\end{proof}

Let $\RR^n$ (and thus its complexification $\CC^n$) be endowed with a linear action of the group $G$. For a
$G$-invariant analytic germ $f: (\CC^n,0) \to (\CC,0)$ with an isolated critical point at the origin,
its Milnor fibre is
\[ M_{f,\eps}= f^{-1}(\eps) \cap B_{\delta}^{2n},
\]
where $0<\Vert\eps\Vert\ll\delta$ are small enough, $B^{2n}_{\delta}$ is the ball of radius $\delta$
centred at the origin in $\CC^n$. It has the homotopy type of a bouquet of $(n-1)$-dimensional spheres.

Let the germ $f$ be real, that is, it takes real values on $(\RR^n,0) \subset (\CC^n,0)$. In this case
one can define the real Milnor fibre (or rather fibres) of the germ $f$. For $\eps$ real (small enough) let 
\[ M_{f,\eps}^\RR = f^{-1}(\eps) \cap B_\delta^{2n} \cap \RR^n
\]
be the real part of the Milnor fibre $M_{f,\eps}$. It is a $C^\infty$-manifold of real dimension $(n-1)$
with boundary. One can see that, as $C^\infty$-manifolds, the manifolds $M_{f,\eps}^\RR$ are the same
for positive $\eps$ and also the same for negative $\eps$. Thus there exist essentially two real Milnor
fibres: $M^+_{f}$ and $M^-_{f}$. The Milnor fibres $M^\pm_{f}$ carry actions of the group $G$. One can show
that the equivariant index ${\rm ind}^G(df; \RR^n,0)$ of the differential $df$ is equal to minus the
reduced equivariant Euler characteristic $\overline{\chi}^G(M^-_{f})=\chi^G(M^-_{f})-1$ of the ``negative''
Milnor fibre: \cite[Proposition~4.11]{EG-EJM}. The function $f$ induces an analytic function $\check{f}$
on $\CC^n/G$ such that $f=\check{f} \circ \pi$.

Proposition~\ref{prop:indquot} gives
 \begin{equation}
  \ind(d\check{f}; \RR^n/G,0)=r^{(0)}\ind^G(df; \RR^n,0).
 \end{equation}

%%%%%%%%%%%%%%%%%%%%%%%%%
\section{Real quotient singularities} \label{sect:realquot}
%%%%%%%%%%%%%%%%%%%%%%%%%%%%
For a $G$-invariant analytic 1-form $\omega= \sum_{i=1}^n A_i(\xx)dx_i$ ($\xx:=(x_1, \ldots , x_n)$) on $(\RR^n,0)$, let 
\[
\Omega_\omega := \Omega^n_{\RR^n,0}/\omega \wedge \Omega^{n-1}_{\RR^n,0} 
\]
and let 
\[
\Omega^\CC_\omega := \Omega^n_{\CC^n,0}/\omega \wedge \Omega^{n-1}_{\CC^n,0} \, .
\]
The residue pairing 
\[ 
B_\omega^\CC : \Omega^\CC_\omega \otimes_\CC \Omega^\CC_\omega \to \CC
\]
is defined by
\begin{equation*}
B_\omega^\CC(\zeta_1,\zeta_2) =   {\rm Res} \left[ 
\begin{array}{c} \varphi_1(\xx) \varphi_2(\xx) d\xx \\
A_1 \cdots A_n \end{array} \right] 
 =  \frac{1}{(2 \pi i)^n} \int \frac{\varphi_1(\xx) \varphi_2(\xx)}{A_1 \cdots A_n} d\xx,
\end{equation*}
where $d\xx:=dx_1 \wedge \cdots \wedge dx_n$,  $\zeta_i=\varphi_i(\xx) d\xx$ for $i=1,2$ and the integration is along
the cycle given by the equations $\Vert A_k(\xx)\Vert=\delta_k$ with positive $\delta_k$ small enough. If the
1-form $\omega$ is real, the restriction of the pairing $B^\CC_\omega$ to $\Omega_\omega$ gives the
(real) residue pairing
\[
B_\omega : \Omega_\omega \otimes_\RR \Omega_\omega \to \RR \, .
\]
Let $B^G_\omega : \Omega^G_\omega \otimes_\RR \Omega^G_\omega \to \RR$ be its restriction to the
$G$-invariant part $\Omega^G_\omega$ of $\Omega_\omega$. It is non-degenerate as well.

Note that the index of a complex valued 1-form $\omega$ on $(V,0)$ differs by the sign $(-1)^n$ from the index of
its real part ${\rm Re}\, \omega$ on $(V,0)$ \cite[Remark~2.3]{EGS} (see also \cite{EG-BBMS}).
%%Comment: From Section 8
Therefore, the (complex) equivariant index  ${\rm ind}^G(\omega; \CC^n,0)$ is $(-1)^n$ times the index
${\rm ind}^G({\rm Re} \,\omega; \CC^n,0)$. It is possible to show that the image of the index
${\rm ind}^G(\omega; \CC^n,0)$ under the map $r:A(G) \to R_\CC(G)$ is equal to the class
$[\Omega^\CC_\omega]$ of the $G$-module $\Omega^\CC_\omega$: \cite{GZM}. Therefore, for the $G$-invariant
part $\left(\Omega^{\CC}_\omega\right)^G$ of $\Omega^\CC_\omega$, one has 
\[
\dim \left(\Omega^{\CC}_\omega\right)^G = r^{(0)}{\rm ind}^G(\omega; \CC^n, 0) \, .
\]
Taking into account relations between dimensions of modules in the complex case and signatures of
quadratic forms in the real case in the Eisenbud--Levine--Khimshiashvili theory, one might expect that
\[
{\rm sgn}\, B_\omega^G = r^{(0)}{\rm ind}^G(\omega; \RR^n, 0) \, .
\]
However, in general this is not the case.

\begin{example} Let $\omega= df$, where $f(\xx)=x_1^2+ \cdots + x_n^2$ is a $\ZZ_2$-invariant function on
$\RR^n$ considered with the action $\sigma \xx= -\xx$ ($\sigma$ is the generator
of $\ZZ_2$). One can see that $r^{(0)}{\rm ind}^{\ZZ_2}(\omega; \RR^n,0)=1$ (since the 1-form $\omega$ is
radial at the origin). The module $\Omega_\omega$ is generated by the 1-form $d\xx$
which is not $\ZZ_2$-invariant for $n$ odd. Therefore, in this case
$\Omega^G_\omega=0$ and ${\rm sgn}\, B_\omega^G=0$.
\end{example}

A reason for the difference between ${\rm sgn}\, B_\omega^G$ and $r^{(0)}{\rm ind}^G(\omega; \RR^n, 0)$
is the fact that a computation of ${\rm sgn}\, B_\omega^G$ embraces singular points of a deformation of
the 1-form $\omega$ which are not real, but become real after factorization by the group $G$. The latter
means that the $G$-orbit of such a point is mapped into itself by the complex conjugation.

\begin{example}
For the action of $\ZZ_2$ on $\RR^n$ given by $\sigma \xx=-\xx$, the quotient $\RR^n/\ZZ_2$ is the
semialgebraic variety defined by 
\[
u^2_{ij}=u_{ii} \cdot u_{jj} \mbox{ for } 1 \leq i < j \leq n, \quad u_{ii} \geq 0
\mbox{ for } 1 \leq i \leq n,
\]
(here $u_{ij}=x_ix_j$). Its Zariski closure is the cone $u^2_{ij}=u_{ii} \cdot u_{jj}$ (without the
inequalities). It is the image under the quotient map $\CC^n \to \CC^n/\ZZ_2$ of the subset
$\RR^n \cup i\RR^n \subset \CC^n$ (a union of vector subspaces).
\end{example}

Let us describe the Zariski closure $\overline{\RR^n/G}$ of $\RR^n/G$, or rather its preimage under
the quotient map $\CC^n \to \CC^n/G$. It is shown in \cite[Section~2]{Hu}, that, if the order of $G$ is odd,
then $\overline{\RR^n/G}= \RR^n/G$. Otherwise, for an element $g \in G$ of even order, let
$\RR^n_{g\pm} := \{ \xx \in \RR^n \, | \, g\xx=\pm \xx\}$. The subspace $\RR^n_{g\pm}$ can also be described
in the following way. Let $\RR^n= \bigoplus_\alpha \RR^n_\alpha$ be the decomposition of the $G$-module
$\RR^n$ into the submodules corresponding to the different irreducible representations $\alpha$ of the
group $G$. Then $\RR^n_{g\pm}$ is the direct sum of the components $\RR^n_\alpha$ over representations
$\alpha$ such that $\alpha(g)$ is the multiplication by $(\pm1)$.
%% Let $\RR^n_{g+}$ be the sum of the components $\RR^n_\alpha$ such that $\alpha(g)$ is the multiplication by 1.

\begin{proposition} \label{prop:quotient}
One has
\[ 
\pi^{-1}(\overline{\RR^n/G}) = 
\bigcup_{g \in G \atop g \,{\rm of} \,{\rm  even} \, {\rm order}} (\RR^n_{g+} \oplus i\RR^n_{g-}) \, .
\]
\end{proposition}

\begin{proof} Let $p \in \CC^n$ be such that $Gp=G\overline{p}$, that is, the complex conjugate
$\overline{p}$ of $p$ satisfies $\overline{p}=gp$ for a certain $g \in G$ (in this case $g$ is of even order).
The vector $u=\frac{1}{2}(p + \overline{p})$ is real and satisfies the condition $gu=u$, that is,
$u \in \RR^n_{g+}$. The vector $v=\frac{1}{2i}(p-\overline{p})$ is also real and $gv=-v$, that is,
$v \in \RR^n_{g-}$. Therefore $p=u+iv \in \RR^n_{g+} \oplus i\RR^n_{g-}$.
\end{proof}

%% Let $\omega$ be a real analytic 1-form on $\CC^n/G$ (that is, it is real on $\RR^n/G$ and therefore
%% on $\overline{\RR^n/G}$ as well). We are now ready to state the main result of the paper.
If $\omega$ is a real analytic $G$-invariant 1-form on $\CC^n$ (that is, it is real on $\RR^n$), it is
real on $\pi^{-1}(\overline{\RR^n/G})$ as well. We are now ready to state the main result of the paper.

\begin{theorem} \label{thm:main}
For a real analytic $G$-invariant 1-form $\omega$ one has
\begin{equation} \label{eqn:main}
{\rm sgn}\, B_\omega^G = r^{(0)}{\rm ind}^G(\omega; \pi^{-1}(\overline{\RR^n/G}),0) \, .
\end{equation}
In particular, for a $G$-invariant analytic function $f$ on $\RR^n$ and for its push-forward $\check{f}$ on
$\CC^n/G$ one has
\begin{equation} \label{eqn:main2}
{\rm sgn}\, B_{df}^G = {\rm ind}(d\check{f}, \overline{\RR^n/G}, 0) \, .
\end{equation}
\end{theorem}

The proof will be given in Sections~\ref{sect:laws}--\ref{sect:one-dimrep} and starts from
the following statement.

\begin{proposition} \label{prop:morse}
There exists a real $G$-invariant deformation $\widetilde{\omega}$ of the 1-form $\omega$ such that
$\widetilde{\omega}$ has only non-degenerate singular points in $\CC^n$ (in a neighbourhood of the origin).
\end{proposition}

\begin{proof}
Let $j: \CC^n/G \to \CC^N$ be a real embedding of the quotient into the affine space with the coordinates
$z_1, \ldots , z_N$. (It is given by $z_j=\varphi_j(\xx)$, where $\varphi_j$ are generators of the algebra
of $G$-invariant functions on $\CC^n$.) For generic real $\lambda_1, \ldots , \lambda_N$ the 1-form
$\widetilde{\omega} = \omega + t \cdot \left( \sum_{j=1}^N \lambda_j dz_j \right)$ possesses the required
property.
\end{proof}

%%%%%%%%%%%%%%%%%%%%%%%%%%%%%%
\section{Laws of conservation of numbers} \label{sect:laws}
%%%%%%%%%%%%%%%%%%%%%%%%%%%%%%
Let $\widetilde{\omega}$ b a real deformation of the 1-form $\omega$ (not necessarily given by
Proposition~\ref{prop:morse}). The set ${\rm Sing}\, \widetilde{\omega}$ of singular points of the 1-form
$\widetilde{\omega}$ is a $G$-set. For a singular point $p \in {\rm Sing}\, \widetilde{\omega}$ let $G_p$
be its isotropy subgroup.

The equivariant index satisfies the following law of conservation of number (see \cite[p.~295]{EG-EJM}).

\begin{proposition} \label{prop:law-index}
One has
\[
{\rm ind}^G(\omega; \pi^{-1}(\overline{\RR^n/G}),0) = 
\sum_{[p] \in {\rm Sing}\, \widetilde{\omega}/G}
{\rm I}^G_{G_p} ({\rm ind}^{G_p}(\widetilde{\omega}; \pi^{-1}(\overline{\RR^n/G}),p)) \, ,
\]
where the sum is over all the orbits of $G$ on ${\rm Sing}\, \widetilde{\omega}$, $p$ is a representative
of the orbit $[p]$, ${\rm I}^G_{G_p}$ is the induction map  $A(G_p) \to A(G)$ $($sending $[G_p/H]$ to $[G/H]$$)$.
\end{proposition}

Since $r^{(0)} {\rm I}^G_{G_p} = r^{(0)}$, one has the following statement.
%%Comment {\rm I} in previous papers

\begin{corollary} \label{cor:law-index}
One has
\[
r^{(0)}{\rm ind}^G(\omega; \pi^{-1}(\overline{\RR^n/G}),0) =
\sum_{[p] \in {\rm Sing}\, \widetilde{\omega}/G}
r^{(0)}{\rm ind}^{G_p}(\widetilde{\omega}; \pi^{-1}(\overline{\RR^n/G}),p) \, .
\]
\end{corollary}

One has a similar law of conservation of number for the signature of the residue pairing on the
$G$-invariant part of $\Omega_\omega$.

\begin{theorem} \label{thm:law-sgn}
One has
\[
{\rm sgn} \, B^G_\omega = \sum_{[p] \in {\rm Sing}\, \widetilde{\omega}/G}
{\rm sgn} \, B^{G_p}_{\widetilde{\omega},p} \, ,
\]
where the signature in a summand on the right hand side is computed at the point $p$.
\end{theorem}

\begin{proof} 
The possibility to deform the 1-form $\widetilde{\omega}$ to a one with only non-degenerate singular
points permits us to assume that already $\widetilde{\omega}$ has this property. It is known that
the bilinear form $B_\omega$ is the limit (when the deformation parameter tends to zero) of the following
bilinear form $B_{\widetilde{\omega}}$ on the space $L_{{\rm Sing}\, \widetilde{\omega}}$ of $n$-forms of
the form $\varphi(\xx) d\xx$ where $d\xx=dx_1 \wedge \cdots \wedge dx_n$ and $\varphi$
is a real function on ${\rm Sing}\, \widetilde{\omega}$. (A function $\varphi$ on 
${\rm Sing}\, \widetilde{\omega}$ is called {\em real} if $\varphi(\overline{z})=\overline{\varphi(z)}$
for all  $z \in {\rm Sing}\, \widetilde{\omega}$.) For two real functions $\varphi$ and $\psi$, the value
of $B_{\widetilde{\omega}}$ is defined by
\begin{equation}\label{eqn:limit}
B_{\widetilde{\omega}}(\varphi(\xx) d\xx, \psi(\xx) d\xx) =
\sum_{z \in {\rm Sing}\, \widetilde{\omega}} \frac{\varphi(z)\psi(z)}{\calJ_{\widetilde{\omega}}(z)} \,,
\end{equation}
where $\calJ_{\widetilde{\omega}}$ is the Jacobian of the 1-form $\widetilde{\omega}$:
if $\widetilde{\omega}=\sum_{i=1}^n A_i(\xx) dx_i$, then
$\calJ_{\widetilde{\omega}}(\xx) = \det \left( \frac{\partial A_i(\xx)}{\partial x_j} \right)$.

Consider the action of the complex conjugation on ${\rm Sing}\, \widetilde{\omega}$. It commutes with
the action of $G$. Let $\calG$ be the group generated by $G$ and by the complex conjugation (the direct
product of $G$ and $\ZZ_2$). The set ${\rm Sing}\, \widetilde{\omega}$ is the disjoint union of the orbits
${\calG}p$ for representatives of the classes $[p] \in {\rm Sing}\, \widetilde{\omega}/\calG$. Each orbit
$\calG p$ either coincides with $G p$ (if $\overline{p} \in G p$) or is the disjoint union of $Gp$ and
$G{\overline{p}}$ (if $\overline{p} \not\in Gp$). The bilinear form $B_{\widetilde{\omega}}$ is the
direct sum of its restrictions to the corresponding subspaces
$L_{\calG p} \subset L_{{\rm Sing}\, \widetilde{\omega}}$  for $[p] \in {\rm Sing}\, \widetilde{\omega}/\calG$
and the action of $G$ preserves these spaces.
Let $\det: G \to \ZZ_2$ be the natural (determinant) homomorphism. (The action of an element $g \in G$
belongs to ${\rm SL}(n; \RR)$ if and only if $\det g = 1$.) Let $K$ be the kernel of $\det$.

Let $\calG p=Gp \sqcup G\overline{p}$. If the isotropy subgroup $G_p$ of the point $p$ is not contained
in $K$, then $L^G_{\calG p}=0$ and thus the signature of the corresponding summand is equal to zero.
If $G_p \subset K$, the space $L^G_{\calG p}$ is two-dimensional and is generated by the forms
$\varphi_{\rm Re} d\xx$ and $\varphi_{\rm Im} d\xx$ where the function
$\varphi_{\rm Re}$ has values $\pm 1$ on the points of $\calG p$ (so that $\varphi_{\rm Re}(p)=1$,
$\varphi_{\rm Re}(gz)= (\det g) \varphi_{\rm Re}(z)$, $\varphi_{\rm Re}(\overline{z})=\varphi_{\rm Re}(z)$)
and the function $\varphi_{\rm Im}$ has values $\pm i$ (so that
$\varphi_{\rm Im}(p)=i$, $\varphi_{\rm Im}(gz)=
(\det g)\varphi_{\rm Im}(\overline{z})$, $\varphi_{\rm Im}(\overline{z})=-\varphi_{\rm Im}(z)$). 
Equation~\ref{eqn:limit} gives the following matrix of the pairing $B_{\widetilde{\omega}}$ on these elements
\[
\left( \begin{array}{cc} m(\calJ^{-1} + \overline{\calJ}^{-1}) & m (\calJ^{-1} - \overline{\calJ}^{-1})i\\
m(\calJ^{-1} - \overline{\calJ}^{-1})i & m(\calJ^{-1} + \overline{\calJ}^{-1}) \end{array} \right) ,
\]
where $m$ is the number of elements in the orbit $Gp$, $\calJ:=\calJ(p)$. The signature of this matrix
is equal to zero and therefore pairs of complex conjugate $G$-orbits do not contribute to
${\rm sgn}\, B^G_{\widetilde{\omega}}$.

Let $\calG=Gp$. If $G_p \not\subset K$, then $L^G_{\calG p}=0$ as above and thus it gives zero impact to
${\rm sgn}\, B^G_{\widetilde{\omega}}$. The isotropy group $G_p$ acts non-trivially on $1 d\xx$
and therefore ${\rm sgn}\, B^G_{\widetilde{\omega},p}=0$.

Assume now that $\calG=Gp$ and $G_p \subset K$. Let $\overline{p}=g_0p$, $g_0 \in G$, that is,
$p \in i\RR^n_{g_0-} \oplus \RR^n_{g_0+}$ in terms of Proposition~\ref{prop:quotient}. The space
$\RR^n_{g_0-}$ is even-dimensional if $\det g_0=1$ and is odd-dimensional if $\det g_0=-1$. The space
$L^G_{\calG p}$ is one-dimensional. If $\det g_0=1$ (or $\det g_0=-1$), then the space $L^G_{\calG p}$
is generated by the $n$-form $\varphi d\xx$, where the (real) function $\varphi$ on $G_p$ is such
that $\varphi(gz)=(\det g)\varphi(z)$ for $g \in G$ and it has values $\pm 1$ (or $\pm i$) on the points
of $Gp$ so that $\varphi(p)=1$, $\varphi(\overline{p})=1$ (or $\varphi(p)=i$, $\varphi(\overline{p})=-i$
respectively). If $\det g_0=1$, then the value of the quadratic form $B_{\widetilde{\omega}}$ on $\varphi$
is equal to $\frac{m}{\calJ(p)}$, where $m=|Gp|$. Otherwise (if $\det g_0=-1$) it is equal to
$-\frac{m}{\calJ(p)}$. Real coordinates on $\pi^{-1}(\overline{\RR^n/G})$ at the point $p$ are
$ix_1, \ldots , ix_k, x_{k+1}, \ldots , x_n$, where $x_1, \ldots x_k$ are the coordinates on $\RR^n_{g_0-}$
and $x_{k+1}, \ldots , x_n$ are the coordinates on $\RR^n_{g_0+}$. The value of the Jacobian of
$\widetilde{\omega}$ in these coordinates differs from the value of the Jacobian in the
coordinates $x_1, \ldots , x_n$ by the sign $(-1)^k$. Thus the value of the quadratic form
$B_{\widetilde{\omega},p}$ on the ($G_p$-invariant!) $n$-form $d\xx$ is equal
to $(-1)^k \frac{1}{\calJ(p)}$. Thus in this case the impact of $L^G_{\calG p}$ to the signature of
$B_{\widetilde{\omega}}$ also coincides with the signature of $B_{\widetilde{\omega},p}$.
\end{proof}

%%%%%%%%%%%%%%%%%%%%%%%%
\section{Reduction to one-dimensional representations} \label{sect:reduction}
%%%%%%%%%%%%%%%%%%%%%%%%%%%%%%%%%
Due to Theorem~\ref{thm:law-sgn}, the statement of Theorem~\ref{thm:main} can be verified for $G$-invariant
1-forms which are non-degenerate at the origin. Deforming the 1-form in the class of non-degenerate (and
$G$-invariant) ones, we can assume that the 1-form $\omega$ has the shape
\[
\omega= \sum_{i,j=1}^n \ell_{ij} x_j dx_i.
\]
(We shall call 1-forms of this type {\em linear} ones.) 

Assume that 
\begin{equation} \label{eqn:alpha}
\RR^n = \bigoplus_\alpha \RR^n_\alpha
\end{equation}
is the decomposition of the $G$-module $\RR^n$ into parts corresponding to different irreducible
representations $\alpha$ of the group $G$. (Each representation $\alpha$ is either one- or two-dimensional.)

\begin{proposition} \label{prop:schur}
The 1-form $\omega$ is the direct sum of 1-forms $\omega_\alpha$ on $\RR^n_\alpha$.
\end{proposition}

\begin{proof} 
Assume that $\sum x_i^2$ is a $G$-invariant quadratic form on $\RR^n$. Then the mapping $\RR^n \to \RR^n$
defined by 
\[ 
(x_1, \ldots , x_n) \mapsto (\sum \ell_{1j} x_j, \ldots , \sum \ell_{nj} x_j)
\]
is a $G$-invariant operator. According to Schur's lemma, this operator is the direct sum of operators on
the subspaces $\RR^n_\alpha$.
\end{proof}

\begin{proposition} \label{prop:connected}
The space of $G$-invariant non-degenerate linear 1-forms on the subspace $\RR^n_\alpha$ is connected if
$\dim \alpha=2$ and has two components if $\dim \alpha=1$ and $\RR^n_\alpha \neq 0$.
\end{proposition}

\begin{proof} We shall identify linear 1-forms with operators as in Proposition~\ref{prop:schur}. If the
representation $\alpha$ is one-dimensional, then all operators $\RR^n_\alpha \to \RR^n_\alpha$ are
$G$-equivariant. Thus the space of non-degenerate $G$-invariant linear 1-forms on $\RR^n_\alpha$ can be
identified with ${\rm GL}(s,\RR)$, where $s=\dim\RR^n_\alpha$
%% for some $s \geq 0$ 
and it has two connected components if $s \neq 0$.
The complexification of a two-dimensional representation $\alpha$ is the sum of two one-dimensional complex
representations $\beta$ and $\overline{\beta}$. A $G$-equivariant operator from the space of the
representation $\alpha$ to itself splits into operators from $\beta$ to itself (given by the multiplication
by $z_1$) and from $\overline{\beta}$ to itself (given by the multiplication by $z_2$). The fact that
the operator is real yields that $z_2=\overline{z_1}$. Therefore the space of operators from $\alpha$
to itself can be identified with $\CC$ and the space of non-degenerate $G$-invariant 1-forms on $\RR^n_\alpha$
can be identified with ${\rm GL}(s,\CC)$ ($s=\dim\RR^n_\alpha/2$) which is connected.
\end{proof}

Propositions~\ref{prop:schur} and \ref{prop:connected} imply that it is sufficient to verify the statement
of Theorem~\ref{thm:main} for a 1-form $\omega$ such that its restriction to $\RR^n_\alpha$ is $d(\sum x_i^2)$
if $\alpha$ is two-dimensional and is $d(\sum \pm x_i^2)$ if $\alpha$ is one-dimensional. (In the latter case
one can assume that the number of minus signs is $\leq 1$.)

Let, as above, $\omega=\bigoplus_\alpha \omega_\alpha$, where $\omega_\alpha$ is of type $d(\sum \pm x_i^2)$
(with only plus signs for two-dimensional representations). Let $\omega'$ be the direct sum of the 1-forms
$\omega_\alpha$ with one-dimensional $\alpha$ (defined on the direct sum $(\RR^n)'$ of the subspaces
$\RR^n_\alpha$ with $\dim \alpha =1$).

\begin{proposition} \label{prop:sgn'}
One has
\begin{eqnarray}
{\rm sgn} \, B^G_{\omega'} & = & {\rm sgn} \, B^G_\omega, \label{eqn:sgn'} \\ 
r^{(0)}{\rm ind}^G(\omega'; \pi^{-1}(\overline{(\RR^n)'/G}),0) &=& 
r^{(0)}{\rm ind}^G(\omega; \pi^{-1}(\overline{\RR^n/G}),0) \, . \label{eqn:index'}
\end{eqnarray}
\end{proposition}

\begin{proof} Equation~(\ref{eqn:sgn'}) is obvious.

To show (\ref{eqn:index'}), we shall prove a somewhat stronger statement:
\[
{\rm ind}^G(\omega'; \pi^{-1}(\overline{(\RR^n)'/G}),0) = 
{\rm ind}^G(\omega; \pi^{-1}(\overline{\RR^n/G}),0) \, .
\]
The complement $\pi^{-1}(\overline{\RR^n/G}) \setminus \pi^{-1}(\overline{(\RR^n)'/G}),0)$ is the disjoint
union of ($G$-invariant) strata of the form
\[
\Xi = \prod_{{\rm some}\, \beta: \atop \dim \beta=2}
(\RR^n_\beta)^\ast \times \prod_{{\rm some}\, \beta: \atop \dim \beta=2} (i\RR^n_\beta)^\ast \times
\prod_{{\rm some}\, \alpha: \atop \dim \alpha=1} (\RR^n_\alpha)^\ast \times
\prod_{{\rm some}\, \alpha: \atop \dim \alpha=1} (i\RR^n_\alpha)^\ast \, ,
\]
where $(\RR^n_\bullet)^\ast:= \RR^n_\bullet \setminus \{ 0 \}$.

We have the equality 
\[
{\rm ind}^G(\omega; \pi^{-1}(\overline{\RR^n/G}),0) =
{\rm ind}^G(\omega'; \pi^{-1}(\overline{(\RR^n)'/G}),0) + \sum _\Xi {\rm ind}^G(\omega; \Xi,0) \, .
\]
We shall prove that ${\rm ind}^G(\omega; \Xi,0)=0$ for each $\Xi$. Assume that this is already proven
for strata of lower dimensions. The equivariant index ${\rm ind}^G(\omega; \overline{\Xi},0)$ on the closure
\[
\overline{\Xi} = \prod_{{\rm some}\, \beta: \atop \dim \beta=2}
\RR^n_\beta \times \prod_{{\rm some}\, \beta: \atop \dim \beta=2} i\RR^n_\beta \times
\prod_{{\rm some}\, \alpha: \atop \dim \alpha=1} \RR^n_\alpha \times
\prod_{{\rm some}\, \alpha: \atop \dim \alpha=1} i\RR^n_\alpha
\]
of the stratum $\Xi$  satisfies the equality
\begin{equation} \label{eqn:xi}
{\rm ind}^G(\omega; \overline{\Xi},0) = 
{\rm ind}^G(\omega; \overline{\Xi}',0) + \sum_{\Sigma \subsetneq \overline{\Xi}}
{\rm ind}^G(\omega; \Sigma,0) + {\rm ind}^G(\omega; \Xi,0) \, ,
\end{equation} 
where 
\[ 
\overline{\Xi}' = \prod_{{\rm some}\, \alpha: \atop \dim \alpha=1}
\RR^n_\alpha \times \prod_{{\rm some}\, \alpha: \atop \dim \alpha=1} i\RR^n_\alpha 
\]
and the sum is over all strata of lower dimensions inside $\overline{\Xi}$. This sum is assumed to be equal
to zero. The latter summand is equal to $k[G/G_\Xi]$, where $k$ is an integer and $G_\Xi$ is the isotropy
group of each point of $\Xi$. The natural homomorphism $A(G) \to \ZZ$ sends the class of a $G$-set to its
number of elements. The number of elements of ${\rm ind}^G(\omega; \overline{\Xi},0)$ or of
${\rm ind}^G(\omega; \overline{\Xi}',0)$ is the usual (integer valued) index of $\omega$ on $\overline{\Xi}$
or on $\overline{\Xi}'$ respectively. These two (usual) indices coincide since the 1-form $\omega$ on
$\prod \RR^n_\beta \times \prod i\RR^n_\beta$ is the differential of a quadratic function with even
numbers of plus and minus signs. Therefore $k=0$ and  ${\rm ind}^G(\omega; \Xi,0)=0$.
\end{proof}

Proposition~\ref{prop:sgn'} permits to verify the statement of Theorem~\ref{thm:main} for the space
$\RR^n$ with only one-dimensional irreducible representations of $G$. We shall show that one can assume
in addition that $\RR^n$ does not contain the trivial representation ${\bf 1}$.
Let
\[
(\RR^n)'':= \bigoplus_{\dim \alpha =1 \atop \alpha \neq {\bf 1}} \RR^n_\alpha
\]
and let $\omega'':=\omega|_{(\RR^n)''}$. 

\begin{proposition} \label{prop:sgn''}
If $\omega_{\bf 1}=d(\sum x_i^2)$, then one has 
\begin{eqnarray}
{\rm sgn} \, B^G_{\omega''} & = & {\rm sgn} \, B^G_\omega, \label{eqn:sgn''} \\ 
r^{(0)}{\rm ind}^G(\omega''; \pi^{-1}(\overline{(\RR^n)''/G}),0) &=& 
r^{(0)}{\rm ind}^G(\omega; \pi^{-1}(\overline{\RR^n/G}),0) \, . \label{eqn:index''}
\end{eqnarray}
\end{proposition}

\begin{proof} Again (\ref{eqn:sgn''}) is obvious. 

The equation 
\[
{\rm ind}^G(\omega''; \pi^{-1}(\overline{(\RR^n)''/G}),0) = 
{\rm ind}^G(\omega; \pi^{-1}(\overline{\RR^n/G}),0) 
\]
follows from the fact that $\pi^{-1}(\overline{(\RR^n)'/G}) =
\pi^{-1}(\overline{(\RR^n)''/G}) \times \RR^n_{\bf 1}$ and the 1-form $\omega_{\bf 1}$ on $\RR^n_{\bf 1}$
is radial.
\end{proof}

\begin{proposition} \label{prop:sgn3}
Let $\omega_{\bf 1}=d(-x_1^2+ \sum_{i>1} x_i^2)$ (that is, it is the differential of a quadratic function
with one minus sign), let $\omega_{\bf 1}':=d(\sum_{i>1} x_i^2)$ and let
$\omega':=\omega_{\bf 1}' \oplus \bigoplus_{\alpha \neq {\bf 1}} \omega_\alpha$. Then one has 
\begin{eqnarray}
{\rm sgn} \, B^G_\omega & = & - {\rm sgn} \, B^G_{\omega'}, \label{eqn:sgn3} \\ 
r^{(0)}{\rm ind}^G(\omega; \pi^{-1}(\overline{\RR^n/G}),0) &=& 
-r^{(0)}{\rm ind}^G(\omega'; \pi^{-1}(\overline{\RR^n/G}),0) \, , \label{eqn:index3}
\end{eqnarray}
and therefore
\[
{\rm sgn} \, B^G_\omega = r^{(0)}{\rm ind}^G(\omega; \pi^{-1}(\overline{\RR^n/G}),0) \, .
\]
\end{proposition}

\begin{proof}
Equation (\ref{eqn:sgn3}) is obvious. 

To prove that
\[
{\rm ind}^G(\omega; \pi^{-1}(\overline{\RR^n/G}),0) =
-{\rm ind}^G(\omega'; \pi^{-1}(\overline{\RR^n/G}),0) \, , 
\]
let us consider the family $\omega_{(t)}$ of 1-forms defined by 
\[
\omega_{(t)} = d(x_1^3-tx_1+ \sum_{i>1} x_i^2) \oplus \bigoplus_{\alpha \neq {\bf 1}} \omega_\alpha \, .
\]
For $t>0$, the 1-form $\omega_t$ has no singular points on $\pi^{-1}(\overline{\RR^n/G})$. For $t<0$ it has
one (non-degenerate) singular point of type $\omega$ and one of type $\omega'$. In both cases
the equivariant indices of the singular points sum up to the equivariant index of $\omega_{(0)}$.
\end{proof}

%%%%%%%%%%%%%%%%%%%%%%
\section{The case of one-dimensional non-trivial representations} \label{sect:one-dimrep}
%%%%%%%%%%%%%%%%%%%%%%%%%%
Propositions~\ref{prop:sgn'}, \ref{prop:sgn''}, and \ref{prop:sgn3} permit to verify the statement of
Theorem~\ref{thm:main} for the space $\RR^n$ with only non-trivial one-dimensional representations and 
\begin{equation} \label{eqn:omega}
\omega =d(\sum \pm x_i^2).
\end{equation}

The first idea would be to compute both sides of Equation~(\ref{eqn:main}) for the forms of the described
type and to compare the results. The left hand side (the signature) is easily computed: it is equal to 0
if the representation of $G$ is not in ${\rm SL}(n; \RR)$ and to $\pm 1= {\rm ind} (\omega,\RR^n,0)$ if the
representation is inside ${\rm SL}(n; \RR)$. The problem is that the computation of the right hand side
(namely of ${\rm ind}^G$: we do not know how it is possible to compute $r^{(0)}{\rm ind}^G$ directly) is
not complicated for a particular case, however it leads to rather involved combinatorial relations for
certain functions on the set of subgroups which we could not understand.

The difficulty to compute ${\rm ind}^G(\omega; \pi^{-1}(\overline{\RR^n/G}),0)$ is related to the fact
that the function $f$ such that $\omega=df$ takes both positive and negative values on
$\pi^{-1}(\overline{\RR^n/G})$. If $\omega=df$ where $f(x)>0$ for
$x \in \pi^{-1}(\overline{\RR^n/G}) \setminus \{ 0 \}$, then
${\rm ind}^G(\omega; \pi^{-1}(\overline{\RR^n/G}),0)=1$. 

\begin{proposition} \label{prop:four}
Let $\omega_0:=d( \sum x_i^4)$. Then one has
\[
{\rm sgn} \, B^G_{\omega_0} = r^{(0)}{\rm ind}^G(\omega_0; \pi^{-1}(\overline{\RR^n/G}),0)=1 \, .
\]
\end{proposition}

\begin{proof}
The equation
\[
{\rm ind}^G(\omega_0; \pi^{-1}(\overline{\RR^n/G}),0)=1
\]
follows from the fact that the function $\sum x_i^4$ is positive on
$\pi^{-1}(\overline{\RR^n/G}) \setminus \{ 0 \}$ (and therefore $\omega_0$ is radial on
$\pi^{-1}(\overline{\RR^n/G})$).

A basis of the space $\Omega_{\omega_0}$ consists of the $n$-forms
$\underline{x}^{\underline{k}}d\xx$, where
$\underline{x}^{\underline{k}}:=x_1^{k_1} \cdots x_n^{k_n}$ with $0 \leq k_i \leq 2$. The Jacobian
$\calJ_{\omega_0}$ of $\omega_0$ is $2^n\prod_{i=1}^n x_i^2$. It is known that the pairing $B_{\omega_0}$
(at least up to an automorphism of $\Omega_{\omega_0}$ defined by the multiplication by an invertible function)
can be computed in the following way. Let $\ell$ be a $G$-invariant linear function on $\Omega_{\omega_0}$
such that $\ell(\calJ_{\omega_0}) >0$. Then 
\[
B_{\omega_0}(\varphi d\xx, \psi d\xx)=\ell(\varphi, \psi) \, .
\]
Let us take the function $\ell$ to be equal to zero on all the elements of the basis except
$\underline{x}^{\underline{2}}d\xx$ ($\underline{2}=(2, \ldots , 2)$) where it is equal to 1.
We have to consider the pairing on the subspace generated by $G$-invariant basis (monomial) $n$-forms.
The $n$-form
$\underline{x}^{\underline{1}}d\xx$ ($\underline{1}=(1, \ldots , 1)$) is $G$-invariant and one has
$B_{\omega_0}(\underline{x}^{\underline{1}} d\xx, \underline{x}^{\underline{1}} d\xx)=1$.
If a basis element $\underline{x}^{\underline{k}}d\xx$, $\underline{k} \neq \underline{1}$, is
$G$-invariant, then the element $\underline{x}^{\underline{2}-\underline{k}}d\xx$ is $G$-invariant
as well and we have 
\[ 
B_{\omega_0}(\underline{x}^{\underline{k}} d\xx, \underline{x}^{\underline{k}} d\xx) =
B_{\omega_0}(\underline{x}^{\underline{2}-\underline{k}} d\xx,
\underline{x}^{\underline{2}-\underline{k}} d\xx)=0,
\quad B_{\omega_0}(\underline{x}^{\underline{k}} d\xx,
\underline{x}^{\underline{2}-\underline{k}} d\xx)=1 \, .
\]
Therefore ${\rm sgn} \, B^G_{\omega_0}=1$.
\end{proof}

Now we can finish the poof of Theorem~\ref{thm:main}. If $G=\{ e\}$ (the trivial group) Theorem~\ref{thm:main}
is just the Eisenbud--Levine--Khimshiashvili theorem. Assume that Theorem~\ref{thm:main} is proved for all
groups of order smaller than the order of $G$. Let $\omega$ be as in (\ref{eqn:omega}).

\begin{proposition} \label{prop:endproof}
Under the assumptions above, one has
\[
{\rm sgn}\, B_\omega^G = r^{(0)}{\rm ind}^G(\omega; \pi^{-1}(\overline{\RR^n/G}),0) \, .
\]
\end{proposition}

\begin{proof}
Let us consider the family $\omega_{(t)}=\omega_0 +t\omega$. For $t>0$ the 1-form $\omega_{(t)}$
has a non-degenerate singular point at the origin of type $\omega$. Theorem~\ref{thm:law-sgn} gives 
\begin{equation}
{\rm sgn} \, B^G_{\omega_0} =
{\rm sgn}\, B_\omega^G + \sum_{p \neq 0} {\rm sgn}\, B_{\omega_{(t)},p}^{G_p} \label{eqn:sgnsplit} 
\end{equation}
and 
\begin{eqnarray}
\lefteqn{r^{(0)}{\rm ind}^G(\omega_0; \pi^{-1}(\overline{\RR^n/G}),0)} \nonumber \\
 & = & r^{(0)}{\rm ind}^G(\omega; \pi^{-1}(\overline{\RR^n/G}),0) +
 \sum_{p \neq 0} r^{(0)}{\rm ind}^{G_p}(\omega_{(t)}; \pi^{-1}(\overline{\RR^n/G}),p)\,. \label{eqn:indsplit}
\end{eqnarray}
By Proposition~\ref{prop:four}, the left hand sides of (\ref{eqn:sgnsplit}) and (\ref{eqn:indsplit})
are equal to each other. By the assumption
\[
{\rm sgn}\, B_{\omega_{(t)},p}^{G_p} = 
r^{(0)}{\rm ind}^{G_p}(\omega_{(t)}; \pi^{-1}(\overline{\RR^n/G}),p) \, .
\]
Therefore 
\[
{\rm sgn} \, B^G_{\omega} = r^{(0)}{\rm ind}^G(\omega; \pi^{-1}(\overline{\RR^n/G}),0) \, .
\]
\end{proof}

%%%%%%%%%%%%%%%%%
\section{Quantum cohomology group and pairings} \label{sect:quantum}
%%%%%%%%%%%%%%%%%%%%%
Let $(\CC^n,0)$ be endowed with an action of a finite abelian group $G$. Without loss of generality
we can assume that the action is linear, that is, it is induced by a representation of $G$.
Let $f:(\CC^n,0) \to (\CC,0)$ be a germ of a $G$-invariant holomorphic function with an isolated
critical point at the origin. For $g \in G$, let $(\CC^n)^g$ be the fixed point set (a vector subspace)
$\{ x \in \CC^n \, | \, gx=x \}$ and let $n_g$ be the dimension of $(\CC^n)^g$. The restriction of $f$
to $(\CC^n)^g$ will be denoted by $f^g$. The germ $f^g: ((\CC^n)^g,0) \to (\CC,0)$ has an isolated critical
point at the origin. Let $M_{f^g}:=M_{f^g,\eps}$ ($\eps$ small enough) be the Milnor fibre of the germ $f^g$.
The group $G$ acts on $M_{f^g}$ and thus on its homology and cohomology groups. 

\begin{definition} (cf.\ \cite{FJR})
The {\em quantum cohomology group} of the pair $(f,G)$ is 
\[
\calH_{f,G}= \bigoplus_{g\in G} \calH_g\,,
\]
where $\calH_g:=H^{n_g-1}(M_{f^g};\CC)^G= H^{n_g-1}(M_{f^g}/G;\CC)$ is the $G$-invariant part of the
vanishing cohomology group $H^{n_g-1}(M_{f^g};\CC)$ of the Milnor fibre of $f^g$. If $n_g=1$, this means
the cohomology group  $\widetilde{H}^0(M_{f^g};\CC)$ reduced modulo a point. If $n_g=0$, one assumes
$H^{-1}(M_{f^g};\CC)$ to be one dimensional  with the trivial action of $G$. (This means that the
``critical point'' of the function of zero variables is considered as a non-degenerate one and thus has
Milnor number equal to one.)
\end{definition}

\begin{remark}
 In~\cite{FJR} the space $\calH_g$ is defined as
 \[
 \calH_g:=H^{n_g}(B_{\delta}^{2n}, M_{f^g}; \CC)^G.
 \]
 However this space
 is canonically isomorphic to $H^{n_g-1}(M_{f^g};\CC)^G$ with the conventions for $n_g=0,1$ in the
 definition above.
 \end{remark}
 
 Let $\Omega^\CC_{f^g} := \Omega^{n_g}_{\CC^{n_g},0}/df^g \wedge \Omega^{n_g-1}_{\CC^{n_g},0}$.
 One has a canonical isomorphism between $\Omega^\CC_{f^g}$ and $H^{n_g-1}(M_{f^g};\CC)$ (for $n_g=0,1$
 as well). (A differential $n_g$-form $\eta$ on $(\CC^{n_g},0)$ gives a well-defined $(n_g-1)$-form
 $\frac{\eta}{df^g}$ on $M_{f^g}$ which corresponds to an element of $H^{n_g-1}(M_{f^g};\CC)$.) This
 isomorphism respects the action of the group $G$ on both spaces and thus defines an isomorphism between
 $(\Omega^{\CC}_{f^g})^G$ and $H^{n_g-1}(M_{f^g};\CC)^G$.
 One also has the residue pairing $B_{f^g}^\CC := B_{df^g}^\CC$ on the space $\Omega^\CC_{f^g}$.
 
 Assume that the function $f$ is quasihomogeneous, that is, there exist positive integers (the {\em weights})
 $w_1, \ldots , w_n$ and $d$ (the {\em quasidegree}) such that 
 \[
 f(\lambda^{w_1}x_1, \ldots , \lambda^{w_n} x_n)=\lambda^df(x_1, \ldots , x_n) \, .
 \]
 The classical monodromy transformation of $f$ is a map $\phi_f$ from the Milnor fibre $M_f=M_{f,\eps}$ to
 itself induced by rotating the value $\eps$ around the the origin in $\CC$ counterclockwise. For a
 quasihomogeneous $f$, the monodromy transformation can be chosen as the restriction to the Milnor fibre
 of the (linear) transformation
\[
J_f(x_1, \ldots , x_n) = (\exp(2\pi i \cdot w_1/d)x_1, \ldots , \exp(2\pi i \cdot w_n/d)x_n) \, .
\]
The transformation $J:=J_f$ preserves the function $f$. Assume that $J \in G$. (This is the condition on
the (abelian) group $G$ to be {\em admissible}: see \cite[Proposition~2.3.5]{FJR}.) In this situation one
also has the following pairing on the space $H^{n-1}(M_f; \ZZ) \cong H^n(\CC^n, M_f; \ZZ)$. Let
$M_f=M_{f+}:=f^{-1}(\eps)$ and $M_{f-}:=f^{-1}(-\eps)$. One has a well-defined pairing
$\langle \ \circ \ \rangle: H_n(\CC^n, M_{f+}; \ZZ) \otimes H_n(\CC^n, M_{f-}; \ZZ) \to \ZZ$
defined by the ''intersection number''. Let $I=\sqrt{J}$ be the linear transformation
\[
I(x_1, \ldots , x_n) = (\exp(\pi i \cdot w_1/d)x_1, \ldots , \exp(\pi i \cdot w_n/d)x_n) \, .
\]
The transformation $I$ sends $M_{f+}$ to $M_{f-}$ and thus defines a map 
\[ 
I_\ast: H_n(\CC^n, M_{f+}; \ZZ)  \to H_n(\CC^n, M_{f-}; \ZZ) \, . 
\]
For $A,B \in H_n(\CC^n, M_f; \ZZ)$ let
\[ 
\langle A,B \rangle := \langle A \circ I_\ast B \rangle \, .
\]
This is a (non-degenerate) pairing (neither symmetric, nor skew-symmetric) on $H_n(\CC^n, M_f; \ZZ)$.
It defines a pairing on the dual space $H^n(\CC^n,M_f; \ZZ)$ and thus on its subspace
$H^n(\CC^n,M_f; \ZZ)^G \cong H^{n-1}(M_f; \ZZ)^G$.

In \cite{FJR} this pairing is identified with the residue pairing on
$(\Omega^{\CC}_f)^G \cong H^{n-1}(M_f; \CC)^G$. However this seems not to be the case.
The residue pairing is symmetric, whereas the pairing $\langle \ , \ \rangle$ on $H^{n-1}(M_f; \CC)^G$
described above is either symmetric or skew-symmetric depending on the number of variables. Indeed,
\[
\langle A, BÊ\rangle := \langle A \circ I_\ast B \rangle =
\langle I_\ast A \circ I^2_\ast B \rangle = \langle I_\ast A \circ J_\ast B \rangle \, .
\]
If the (relative) cycle $B$ is $J_\ast$-invariant (this is the case when $J \in G$), then one has
\[
\langle A, BÊ\rangle = \langle I_\ast A \circ  B \rangle = (-1)^n \langle B \circ I_\ast A \rangle =
(-1)^n \langle B, AÊ\rangle \, .
\]
Thus the pairing $\langle \ ,  \  \rangle$ is symmetric  if $n$ is even and skew-symmetric if $n$ is odd.

%%%%%%%%%%%%%%%%%%%%%%%%
\section{Orbifold index and quantum cohomology group} \label{sect:orbifold}
%%%%%%%%%%%%%%%%%%%%%%%%
Let $(V,0) \subset \CC^N$ be a germ of a complex analytic $n$-dimensional variety with an action of a finite
group $G$ on it. In Section~\ref{sect:equivariant}, we defined the orbifold index
${\rm ind}^{\rm orb}(\omega; X,0)$ for
a $G$-invariant 1-form on $(X,0)$, where $(X,0)$ is a germ of a (real) subanalytic space and
$\omega$ has an isolated singular point at the origin. It can be defined in the same way in the case of
a complex valued 1-form on $(V,0)$ with an isolated singular point at the origin. 
%%Note that the index of a complex valued 1-form $\omega$ on $(V,0)$ differs by the sign $(-1)^n$ from the index of
%%its real part ${\rm Re}\, \omega$ on $(V,0)$ \cite[Remark~2.3]{EGS} (see also \cite{EG-BBMS}.
%% To Section 3

From \cite[Theorem~2]{GZM} one can derive the following result.

\begin{theorem} \label{thm:ind=dim}
One has
\begin{equation*} \label{eqn:ind=dim}
{\rm ind}(\omega; \CC^n/G, 0) = \dim (\Omega_{\omega}^{\CC})^G \, .
\end{equation*}
\end{theorem}

\begin{proof} This follows from \cite[Theorem~2]{GZM} since the $G$-invariant part corresponds to the trivial representation.
\end{proof}

Let $(\CC^n,0)$ be endowed with a linear action of a finite abelian group $G$ and let
$f:(\CC^n,0) \to (\CC,0)$ be a germ of a $G$-invariant holomorphic function with an isolated critical point
at the origin as in Sect.~\ref{sect:quantum}. Let $\check{f}$ be the corresponding function on $(\CC^n/G,0)$
($f=\check{f}\circ\pi$, where $\pi$ is the quotient map $\CC^n\to\CC^n/G$) and let
\[
Q^\CC_{f,G} := \bigoplus_{g \in G} (\Omega_{f^g}^{\CC})^G \, .
\]

\begin{definition}
The {\em orbifold dimension} of $Q^\CC_{f,G}$ is
\[
\dim^{\rm orb} Q^\CC_{f,G} := \sum_{g \in G} (-1)^{n_g} \dim (\Omega_{f^g}^{\CC})^G \, .
\]
\end{definition}

\begin{theorem} \label{thm:complexQ}
One has 
\[
\dim^{\rm orb} Q^\CC_{f,G} = (-1)^n {\rm ind}^{\rm orb}(df; \CC^n,0)= - \overline{\chi}^{\rm orb}(M_f) \, ,
\]
where $\overline{\chi}^{\rm orb}(M_f)$ is the reduced orbifold Euler characteristic of the Milnor fibre of $f$.
\end{theorem}

\begin{proof}
One has 
\begin{equation} \label{eqn:indorb} 
{\rm ind}^{\rm orb}({\rm Re}\, df; \CC^n,0) =
\sum_{g \in G} {\rm ind}({\rm Re}\,d\check{f}^g; (\CC^n)^g/G, 0) \,.
\end{equation}
%% This implies for the index of the complex 1-form $df$ by the remark above:
For the index of the complex 1-form $df$ this implies (taking into account that the index of a complex valued 1-form differs by the sign $(-1)^n$ from the index of
its real part):
\[
(-1)^n {\rm ind}^{\rm orb}(df; \CC^n,0) = \sum_{g \in G} (-1)^{n_g} {\rm ind}(d\check{f}^g; (\CC^n)^g/G, 0) \, .
 \]
 By Theorem~\ref{thm:ind=dim} and the definition of the orbifold dimension of $Q^\CC_{f,G}$, we get
 the first equality. The second equality follows from the equality
 \[
 {\rm ind}^G(df; \CC^n, 0) = (-1)^{n-1}  \overline{\chi}^G(M_f)
 \]
(see \cite[p.~297]{EG-EJM}) by applying the homomorphism $r^{(1)}$ to both sides of the equation.
\end{proof}

Now let $f$ be real and let 
\[ 
Q_{f,G} := \bigoplus_{g \in G} \Omega_{f^g}^G, \quad \calB_{df}:=  \bigoplus_{g \in G}  B_{df^g}^G \, .
\]
Then Theorem~\ref{thm:main} implies the following corollary.

\begin{corollary}
One has
\[
%% {\rm sgn} \, \calB_{df} = {\rm ind}^{\rm orb}(df|_{\RR^n}; \RR^n,0) \, .
{\rm sgn} \, \calB_{df} = {\rm ind}^{\rm orb}(df; \pi^{-1}(\overline{\RR^n/G}),0) \, .
\]
\end{corollary}

From \cite[Proposition~4.11]{EG-EJM} one can derive that
\[
%% {\rm ind}^{\rm orb}(df|_{\RR^n}; \RR^n,0) = - \overline{\chi}^{\rm orb}(M^\RR_{f-}) \, .
{\rm ind}^{\rm orb}(df; \pi^{-1}(\overline{\RR^n/G}),0) = - \overline{\chi}^{\rm orb}(M^-_{f_{\RR}},G) \, ,
\]
where $f_{\RR}$ is the restriction of $f$ to $\pi^{-1}(\overline{\RR^n/G})$ (a real valued function),
$M^-_{f_{\RR}}$ is the ``negative'' Milnor fibre $f_{\RR}^{-1}(-\eps)\cap B_{\delta}$ of the function
$f_{\RR}$ ($\eps>0$ is small enough).

%%%%%%%%%%%%%%%%%%%%%%%%
\section{G-signature and the equivariant index} \label{sect:G-signature}
%%%%%%%%%%%%%%%%%%%%%%%%
If $\omega$ is a real analytic $G$-invariant 1-form on $(\CC^n,0)$ with an isolated singular point at
the origin (in $\CC^n$), the pairing $B_{\omega}$ on
$\Omega_{\omega}=\Omega^n_{\RR^n,0}/\omega \wedge \Omega^{n-1}_{\RR^n,0}$ is also $G$-invariant. In this
situation one has a notion of its $G$-signature ${\rm sgn}^GB_{\omega}$ as an element of the ring $R_{\RR}(G)$
of (real) representations of the group $G$: see, e.~g., \cite{GZ1986}, \cite{Damon}.
One can show that if the order of the group $G$ is odd, the reduction
$r_{\RR}\ind^G(\omega;\RR^n),0$ of the equivariant index of the 1-form $\omega$ under the natural
map $r_{\RR}: A(G)\to R_{\RR}(G)$ is equal to the $G$-signature ${\rm sgn}^GB_{\omega}$ of the
quadratic form $B_{\omega}$. This is essentially proved in \cite{GZ1986}, \cite{Damon}.
(In these papers the statement is formulated in terms of a $G$-degree of a map.)
Together with Theorem~\ref{thm:main} this permits to conjecture that, for a finite abelian group $G$,
the $G$-signature ${\rm sgn}^GB_{\omega}$ might be equal to the reduction of the equivariant index
$\ind^G(\omega;\pi^{-1}(\overline{\RR^n/G}),0)$. This is not the case.

\begin{example}
 Let the group $G=\ZZ_2$ act on $\RR^2$ by $\sigma(x,y)=(-x,-y)$ and let $\omega=df$, where
 $f(x,y)=x^2-y^2$. One has ${\rm sgn}^GB_{\omega}=-1$. The preimage $\pi^{-1}(\overline{\RR^n/G})$
 is $\RR^2\cup i\RR^2$. The equivariant index
 $\ind^G(\omega;\pi^{-1}(\overline{\RR^n/G}),0)$ is equal to 
 $1+\ind^G(\omega;\RR^2\setminus\{0\},0)+\ind^G(\omega;i\RR^2\setminus\{0\},0)$, where
 $\ind^G(\omega;\RR^2\setminus\{0\},0)=c_1[\ZZ_2/(e)]$,
 $\ind^G(\omega;i\RR^2\setminus\{0\},0)=c_2[\ZZ_2/(e)]$,
 $1+2c_1=\ind(\omega;\RR^2,0)=-1$, $1+2c_2=\ind(\omega;i\RR^2,0)=-1$.
 Thus $\ind^G(\omega;\pi^{-1}(\overline{\RR^n/G}),0)=1-2[\ZZ_2/(e)]$
 and $r_{\RR}\ind^G(\omega;\pi^{-1}(\overline{\RR^n/G}),0)=1-2(1+\sigma)=-1-2\sigma$,
 where $\sigma$ is the non-trivial one-dimensional representation of the group $\ZZ_2$.
\end{example}

\bigskip
\noindent Leibniz Universit\"{a}t Hannover, Institut f\"{u}r Algebraische Geometrie,\\
Postfach 6009, D-30060 Hannover, Germany \\
E-mail: ebeling@math.uni-hannover.de\\

\medskip
\noindent Moscow State University, Faculty of Mechanics and Mathematics,\\
Moscow, GSP-1, 119991, Russia\\
E-mail: sabir@mccme.ru
\end{document}